\documentclass[11pt]{amsart}
\usepackage{amssymb,amsmath,amsthm,bbm}

\usepackage{color}

\newtheorem{theorem}{Theorem}[section]

\newtheorem{lemma}[theorem]{Lemma}

\newtheorem{prop}[theorem]{Proposition}
\newtheorem{coro}[theorem]{Corollary}

\theoremstyle{definition}
\newtheorem{definition}[theorem]{Definition}
\newtheorem{defi}[theorem]{Definition}

 \theoremstyle{plain}
\newtheorem*{namedthm}{\namedthmname}
\newcounter{namedthm}

\makeatletter
\newenvironment{named}[1]
  {\def\namedthmname{#1}%
   \refstepcounter{namedthm}%
   \namedthm\def\@currentlabel{#1}}
  {\endnamedthm}
  \makeatother

\newcommand{\RR}{\mathbb{R}}
\newcommand{\R}{\mathbb{R}}
\newcommand{\CC}{\mathbb{C}}
\newcommand{\NN}{\mathbb{N}}

\newcommand{\PSH}{{\rm PSH}}
\newcommand{\setdef}{\; ; \; }

\newcommand{\loc}{{\rm loc}}

\newcommand{\f}{\varphi}
\newcommand{\e}{\varepsilon}

 \numberwithin{equation}{section}

\usepackage{hyperref}
\hypersetup{
    bookmarks=true,         
    unicode=false,         
    pdftoolbar=true,       
    pdfmenubar=true,       
    pdffitwindow=false,    
    pdfstartview={FitH},   
    pdftitle={Pluripotential solutions versus viscosity solutions to complex Monge-Amp\`ere flows},    
    pdfauthor={Vincent Guedj, Chinh H. Lu, and Ahmed Zeriahi},     
    colorlinks=true,       
   linkcolor=black,         
    citecolor=black,        
    filecolor=black,      
    urlcolor=black}          

\author{Vincent Guedj}
\address{Vincent Guedj, Institut de Math\'ematiques de Toulouse  \\ 
Universit\'e de Toulouse, CNRS \\
UPS, 118 route de Narbonne \\
31062 Toulouse cedex 09, France}
\email{\href{mailto:vincent.guedj@math.univ-toulouse.fr}{vincent.guedj@math.univ-toulouse.fr}}
\urladdr{\href{https://www.math.univ-toulouse.fr/~guedj}{https://www.math.univ-toulouse.fr/~guedj/}}

\author{Chinh H. Lu}
\address{Chinh H. Lu, Laboratoire de Math\'ematiques d'Orsay,
 Univ. Paris-Sud,
 CNRS, Universit\'e Paris-Saclay,
  91405 Orsay, France}
\email{\href{mailto:hoang-chinh.lu@math.u-psud.fr}{hoang-chinh.lu@math.u-psud.fr}}
\urladdr{\href{https://www.math.u-psud.fr/~lu/}{https://www.math.u-psud.fr/~lu/}}

\author{Ahmed Zeriahi}
\address{Ahmed Zeriahi, Institut de Math\'ematiques de Toulouse,   \\ Universit\'e de Toulouse, CNRS \\
UPS, 118 route de Narbonne \\
31062 Toulouse cedex 09, France}
\email{\href{mailto:ahmed.zeriahi@math.univ-toulouse.fr}{ahmed.zeriahi@math.univ-toulouse.fr}}
\urladdr{\href{https://www.math.univ-toulouse.fr/~zeriahi/}{https://www.math.univ-toulouse.fr/~zeriahi/}}

\thanks{The authors are partially supported by the ANR project GRACK}
  


\begin{document}

  \title[Complex Monge-Amp\`ere flows]{Pluripotential solutions versus viscosity solutions to complex Monge-Amp\`ere flows}

\setcounter{tocdepth}{1}

 \date{\today}

 \begin{abstract}  
 
We compare various notions of weak subsolutions to degenerate complex Monge-Amp\`ere flows,
showing that they all coincide. This allows us to show  that the viscosity solution coincides with the envelope of pluripotential subsolutions.
\medskip

\noindent {\it Dedicated to Duong Hong Phong on the occasion of his 65th birthday.}
\end{abstract}

 \maketitle

\tableofcontents

\section{Introduction} 

A viscosity approach for parabolic complex Monge-Amp\`ere equations (both in local and global contexts) has been developed in \cite{EGZ15,EGZ16,EGZ18, DLT19}, while  a pluripotential approach has been developed in \cite{GLZ1,GLZ2},  which allows to solve these equations with quite degenerate data. 
The goal of this paper is to compare these two notions, extending
the dictionary established in the elliptic case
(see \cite{EGZ11,HL13,GLZ17}).

Let $\Omega$ be a smooth bounded strictly pseudoconvex domain of $\CC^n$.
We consider  the parabolic complex Monge-Amp\`ere flow in $\Omega_T$
\begin{equation} \label{eq: CMAF}
(dd^c \varphi_t)^n =e^{\dot{\varphi}_t+F(t,z,\f)} g(z) dV(z).
\end{equation}
Here 
\begin{itemize}
\item $T>0$ and $\Omega_T=]0,T[ \times \Omega$ with parabolic boundary $$\partial_0 \Omega_T:= \{0\}\times \Omega \cup [0,T[\times \partial \Omega.$$  
\item $F: [0,T]\times \Omega \times \mathbb{R} \rightarrow \mathbb{R}$ is a continuous function;
\item $dV$ denotes the euclidean volume form in $\CC^n$;
\item $0\leq g$ is a continuous function on $\Omega$;   
\item $(t,x) \mapsto \f(t,x)=\f_t(x)$ is the unknown function and
$\dot{\varphi}_t=\partial_t \f$ denotes the time derivative of $\f$.
\end{itemize}
We assume throughout this article that $h: \partial_0 \Omega_T \rightarrow \mathbb{R}$ is a continuous Cauchy-Dirichlet boundary data, i.e. 
\begin{itemize}
\item $h$ is continuous on $\partial_0 \Omega_T$,  and
\item  $h_0$ is  a continuous plurisubharmonic function in $\Omega$. 
\end{itemize}

	\medskip
	
We first extend the definition of pluripotential subsolutions proposed in \cite{GLZ1}. This new definition 
applies to functions which are not necessarily locally Lipschitz in $t$, it thus allows us to consider 
\eqref{eq: CMAF} for less regular data.

We then show that   these pluripotential parabolic  subsolutions coincide with  viscosity subsolutions:

\begin{named}{Theorem A} \label{thmA}
	Assume $\varphi \in \mathcal{P}  (\Omega_T)$.
The following are equivalent:
 
 \smallskip
 
 (i) $\f$ is a viscosity subsolution to \eqref{eq: CMAF};
 
 \smallskip
 
 (ii) $\f$ is a pluripotential subsolution to \eqref{eq: CMAF}.
\end{named}

Here $  \mathcal{P}  (\Omega_T)$ denotes the set of parabolic potentials, i.e.
locally integrable upper semi-continuous functions $\f$ in $\Omega_T$
whose slices $\f_t=\f(t,\cdot)$ are plurisubharmonic in $\Omega$.

\smallskip

The pluripotential parabolic comparison principle \cite[Theorem 6.5]{GLZ1}
then allows us to  conclude that the envelope of pluripotential subsolutions
is the unique viscosity solution to \eqref{eq: CMAF} :

\begin{named}{Theorem B}\label{thmB}
	Assume that $g>0$ is positive almost everywhere in $\Omega$. 
	Then there is a unique viscosity solution to \eqref{eq: CMAF} with boundary value $h$  which coincides with the envelope of all pluripotential subsolutions. 
\end{named}

The techniques developed in the local context allow us to obtain analogous results in the compact setting, comparing viscosity and pluripotential notions  for complex Monge-Amp\`ere flows
that contain the K\"ahler-Ricci flow as a particular case. These are briefly discussed in Section \ref{sect: compact case}.
 
 \medskip

 \noindent {\bf Acknowledgement.} We thank the referee for useful comments which improve the presentation of the paper. 


  \section{Pluripotential subsolutions}  \label{sec:prelim}
    
      
  Let $\Omega$ be a smoothly bounded strongly pseudoconvex domain in $\CC^n$. By this we mean there exists a smooth strictly 
  plurisubharmonic  function $\rho$ in an open neighborhood of $\bar{\Omega}$ 
  such that $\Omega= \{\rho<0\}$ and $d\rho \neq 0$ on $\partial \Omega$.


 \begin{defi} \label{defi: parabolic potential}
The set of parabolic potentials  $ \mathcal P (\Omega_T)$ consists of upper semicontinuous functions  
$ u : \Omega_T :=  ]0,T[ \times  {\Omega} \longrightarrow [- \infty , + \infty[$  such that
$u \in L^1_{\loc } (\Omega_T)$ and
$\forall t  \in ]0,T[$, the slice $u_t : z \mapsto u (t,z)$ is plurisubharmonic in $\Omega$. 
\end{defi}

Let us stress that -by comparison with \cite{GLZ1}- we do not assume here that 
the family $\{u(\cdot,z) \setdef z \in \Omega\}$ is locally  uniformly  Lipschitz in $]0,T[$. We nevertheless use the same 
notation $ \mathcal P (\Omega_T)$ for the set of parabolic potentials, hoping that no confusion will arise.

A pluripotential subsolution is a parabolic potential  $\f$ that satisfies
$$
(dd^c \varphi)^n \wedge dt \geq e^{\dot{\varphi}_t+F(t,z,\f)} g(z) dV(z) \wedge dt
$$
in the weak sense of (positive) measures in $\Omega_T$.

We need to make sense of all these quantities. The LHS is defined as in \cite{GLZ1} by
using Bedford-Taylor's theory, the novelty here concerns mainly the RHS
as we explain hereafter.

\subsection{Defining the LHS}

The LHS can be defined by using Bedford-Taylor theory:

\begin{lemma}
If $u \in \mathcal{P}(\Omega_T) \cap L^{\infty}_{\loc} (\Omega_T)$ then $dt \wedge (dd^c u_t)^n$ is well-defined as a positive Borel measure in $\Omega_T$. 
\end{lemma}

\begin{proof}
Fix $\chi$ a test function in $\Omega_T$ with support contained in $J \times D \Subset \Omega_T$. We regularize $u$ by taking sup convolution: 
for $\ (t,z) \in J \times D$ we set
$$
u^{j}(t,z) : = \sup \{ u(s,z) - j^2(t-s)^2 \setdef s \in ]0,T[  \, \} . 
$$

The functions $u^j$ decrease pointwise to $u$ on $J \times D$ (by upper semi-continuity of $u$). 
Since $t \mapsto u^j$ is continuous,
it follows from \cite[Lemma 2.1]{GLZ1} 
that the function 
$$
t\mapsto \int_{\Omega} \chi(t,z) (dd^c u^j_t)^n 
$$
is continuous in $t$. It follows  from  \cite{BT82} that
$$
\lim_{j\to +\infty} \int_{\Omega} \chi(t,z) (dd^c u^j_t)^n  =  \int_{\Omega }\chi(t,z) (dd^c u_t)^n. 
$$
Taking  limits as $j\to +\infty$ we obtain that $t\mapsto \int_{\Omega }\chi(t,z) (dd^c u_t)^n$ is a bounded Borel measurable function in $]0,T[$. 
The Chern-Levine-Nirenberg  inequalities yield
$$
\left | \int_{\Omega_T} \chi(t,z) dt \wedge (dd^c u_t)^n \right | \leq  C(J,D,u)  \sup_{\Omega_T}| \chi |, 
$$
where $C(J,D,u)>0$ is a constant. It thus follows that the distribution $dt \wedge (dd^c u_t)^n$ extends as a positive Borel measure in $\Omega_T$. 
\end{proof}
   
%

\subsection{Defining the RHS}

For each $u\in \mathcal{P}(\Omega_T)$,  we define $g \partial_t u$ as a distribution on $\Omega_T$ 
by setting
$$
\langle g \partial_t u , \chi \rangle := -\int_{\Omega} \int_0^T \partial_t \chi (t,z) u(t,z) g(z) dt dz,
$$
for all test functions $\chi \in \mathcal{C}^{\infty}(\Omega_T)$ with compact support.

We now wish to interpret the RHS as a supremum of (signed) Radon measures, setting
$$
 e^{\dot{\varphi}_t+F(t,z,\f)} g 
 =g \sup_{a>0} \left\{
  a (\partial_t \varphi + F(t,z,\varphi_t(z)) - a\log a +a
 \right\}.
$$
This relies on
the following observation:

\begin{lemma}\label{lem: new def}
Let $T$ be a positive measure in an open set $D\subset \mathbb{R}^N$, 
$f$ a bounded measurable function on $D$, and $0\leq g\in L^p(D)$. If, for all $a>0$,
$$
T \geq g( a f + a -a\log a) \lambda_N, 
$$ 
in the sense of measures, then $T\geq e^f g$ in the sense of measures in $D$. 
\end{lemma}

Here $\lambda_N$ denotes the Lebesgue measure in $D$.

\begin{proof}
We first assume that $g \geq b>0$ on $D$. Replacing $T$ with $T/g $ we can assume that $g\equiv 1$. 
We regularize $T$ by using non-negative mollifiers, setting
 $T_{\varepsilon}:= T\star \rho_{\varepsilon}$. Then for all $a>0$
$$
T_{\varepsilon} \geq a f\star \rho_{\varepsilon} + a - a\log a,
$$
pointwise on $D$. Taking the supremum over $a>0$ we obtain
$$
T_{\varepsilon} \geq e^{f\star \rho_{\varepsilon}}  
$$
pointwise on $D$.
The inequality thus also holds in the sense of measures. 
Letting $\varepsilon \to 0$ yields the conclusion. 

We now remove the positivity condition on $g$. Since $f$ is bounded, for each $\varepsilon>0$ we can find $c(\varepsilon)>0, A>0$ such that, for all $a\in ]0,A[$, 
$$
T + \varepsilon \lambda_N \geq  (g+ c(\varepsilon))(a f - a \log a +a) \lambda_N, 
$$ 
 It follows from the first step and the fact that $f$ is bounded (so that the supremum can be restricted to $a\in ]0,A[$) that 
$$
T+ \varepsilon \lambda_N \geq (g+ c(\varepsilon)) e^f  \lambda_N
$$
in the sense of measures on $D$. 
The conclusion follows by letting $\e \rightarrow 0$.
\end{proof}

This analysis motivates the following :

  \begin{defi} \label{def: new def no Lip}
Let $u\in  \mathcal P (\Omega_T)\cap L^{\infty}_{\loc }(\Omega_T)$. Then $u$ is a pluripotential subsolution to  \eqref{eq: CMAF} if for all constants $a>0$, 
$$
(dd^c \varphi)^n \wedge dt \geq  g (a (\partial_t \varphi + F(t,z,\varphi_t(z)) - a\log a +a) \, dV(z) \wedge dt
$$
in the sense of distribution in $\Omega_T$.  
%
%
%
\end{defi}


If $u \in \mathcal{P}(\Omega_T) \cap L^{\infty}_{loc} (\Omega_T)$ is locally uniformly semi-concave in $t\in ]0,T[$, then by Lemma \ref{lem: new def} $u$ is a pluripotential subsolution to \eqref{eq: CMAF} iff
$$
(dd^c u_t)^n \geq e^{\partial_t^+ u +F(t,z,u_t)}gdV,
$$
in the sense of Radon measures in $\Omega$. Here $\partial^+_t$ is the right derivative defined pointwise in $\Omega_T$ 
(thanks to the semi-concavity property of $t\mapsto u(t,z)$).  
The above definition thus coincides with the one given in \cite{GLZ1}.

\smallskip

Decreasing limits of pluripotential subsolutions are again subsolutions as the following result shows:

\begin{lemma}\label{lem: convergence of subsolutions}
Let $(u^j)$ be a sequence of pluripotential subsolutions to \eqref{eq: CMAF} which    decreases
to $u\in \mathcal{P}(\Omega_T) \cap L^{\infty}_{\loc}(\Omega_T)$. Then $u$ is a pluripotential subsolution to \eqref{eq: CMAF}. 
\end{lemma}

\begin{proof}
It follows from \cite{BT82}  that  the Radon measures $(dd^c u^j)^n \wedge dt$  weakly converge to 
$ (dd^c u)^n \wedge dt$. 
On the other hand
for each $a>0$
$$
g (a (\partial_t u^j + F) +a -a \log a) \rightarrow g (a (\partial_t u + F) +a -a \log a)
$$
in the weak sense of distributions
in $\Omega_T$. This completes the proof.
\end{proof}

Let us emphasize that in Definition \ref{def: new def no Lip} we do not ask subsolutions to be locally uniformly Lipschitz in $t$ while the definition given in \cite{GLZ1} does assume this regularity. We observe below that the envelopes of subsolutions in both senses do coincide. 
\begin{prop}\label{prop: identification}
	Assume that the data $(F,h,g,u_0)$ satisfy the assumption of \cite{GLZ1}. Let $U$ be the upper envelope of pluripotential subsolutions to \eqref{eq: CMAF} in the sense of Definition \ref{def: new def no Lip}, and $\tilde{U}$ be the envelope of subsolutions  to \eqref{eq: CMAF} in the sense of \cite{GLZ1}. Then $U=\tilde{U}$. 
\end{prop}

\begin{proof}
	By definition we have $\tilde{U} \leq U$. Fix $u$ a pluripotential subsolution  to \eqref{eq: CMAF} in the sense of Definition \ref{def: new def no Lip}. We regularize $u$ by taking convolution  (see \cite{GLZ1})
	$$
	u^{\varepsilon} (t,z) := \int_{\mathbb{R}} u(st, z) \chi((s-1)/\varepsilon) ds,
	$$
	where $\chi$ is a cut-off function. Then $u^{\varepsilon}-c(\varepsilon)(t+1)$ is a pluripotential subsolution to \eqref{eq: CMAF} with data $(F,h,g,u_0)$, where $c(\varepsilon) \to 0$ as $\varepsilon\to 0$. Hence $u^{\varepsilon}-O(\varepsilon)(t+1) \leq \tilde{U}$. Letting $\varepsilon\to 0$ we arrive at $u\leq \tilde{U}$, hence $U\leq \tilde{U}$.  
\end{proof}

\section{Viscosity vs pluripotential subsolutions}

\subsection{Viscosity concepts} 
 
  We now recall the corresponding viscosity notions introduced in \cite{EGZ15}.

 \begin{defi}
Given $u:\Omega_T  \rightarrow \RR$ an u.s.c. bounded function
and $(t_0,x_0) \in X_T$, $q$ is a differential test from above for $u$ at $(t_0,x_0)$ if
\begin{itemize}
\item $q \in {\mathcal C}^{1,2}$ in a small neighborhood $V_0$ of $(t_0,x_0)$;
\item $u \leq q$ in $V_0$ and $u(t_0,x_0)=q(t_0,x_0)$.
\end{itemize}
\end{defi}

\begin{defi}
An u.s.c. bounded function $u:\Omega_T \rightarrow \RR$  is a { \it viscosity subsolution} to  \eqref{eq: CMAF} if for all $(t_0,x_0) \in \Omega_T$ and all differential tests $q$ from above, 
$$
(dd^c q_{t_0}(x_0))^n \geq e^{\dot{q}_{t_0}(x_0)+  F(t_0, x_0,u(t_0,x_0))} g(x_0)dV(x_0).
$$
\end{defi}

Here are few basic facts about viscosity subsolutions:

\begin{itemize}
\item a  ${\mathcal C}^{1,2}$-smooth function is a viscosity subsolution  iff it is psh and a classical subsolution;

\item if $u_1,u_2$ are viscosity subsolutions, then so is $\max(u_1,u_2)$;

\item if $(u_\alpha)_{\alpha \in A}$ is a family of subsolutions which is locally uniformly bounded from above , then
$
\f:=\left( \sup \{ u_\alpha \setdef \alpha \in A \} \right)^*
\text{ is a subsolution};
$

\item If $u$ is a subsolution to $\eqref{eq: CMAF}_g$ then  it is also a subsolution to $\eqref{eq: CMAF}_f$
 with $g$  replaced by $f$, as long  as $0\leq f\leq g$.

\item $u$ is a subsolution to \eqref{eq: CMAF} with $g\equiv 0$ iff $u_t$ is psh for all $t$.
\end{itemize}

\begin{definition}
	A bounded l.s.c. function $u: \Omega_T \rightarrow \mathbb{R}$ is a viscosity supersolution to \eqref{eq: CMAF} if for all $(t_0,z_0) \in \Omega_T$ and all differential tests $q$ from below, 
	$$
	(dd^c q_{t_0}(x_0))_+^n \leq e^{\dot{q}_{t_0}(x_0)+  F(t_0, x_0,u(t_0,x_0))} g(x_0)dV(x_0).
	$$
	Here, for a real  $(1,1)$-form $\alpha$ we define $\alpha_+$ to be $\alpha$ if it is semipositive and $0$ otherwise. 
\end{definition}

\begin{definition}
	A function $u$ is a viscosity solution to \eqref{eq: CMAF} if it is both a viscosity subsolution and a viscosity supersolution to \eqref{eq: CMAF}.
\end{definition}

 Note in particular that viscosity solutions are continuous functions.
 
%

In viscosity theory it is convenient to define the notion of relaxed upper and lower limits of a family of functions.
Let $ \phi^\epsilon : (E,d) \to \R$, $\epsilon > 0$ be a family of locally uniformly bounded functions on a metric space $(E,d)$. We set 
\begin{eqnarray*}
\underline{\phi} (x)  =  {\liminf}_* \, \phi^\epsilon (x) & := & \liminf_{(\epsilon,y) \to (0,x)} \phi^\epsilon (y)\\
\overline{\phi} (x)  =  {\limsup}^* \, \phi^\epsilon (x) & := & \limsup_{(\epsilon,y) \to (0,x)} \phi^\epsilon (y).
\end{eqnarray*}
Observe that $\underline{\phi}$ (resp. $\overline{\phi}$) is lower (resp. upper) semi-continuous on $E$ and  
$\underline{\phi} \leq (\liminf_{\epsilon \to 0^+} \phi^\epsilon)_* $. 
 If the family is constant and equal to $\phi$, $\underline{\phi} = \phi_*$ and $\overline{\phi}= \phi^*$ correpond to the lower and upper semi-continuous regularisations of $\phi$ respectively.

\begin{lemma} \label{lem:Stability} Assume that  $(F^\epsilon)_{0 < \epsilon < \epsilon_0}$ is a family of continuous functions on $]0,T[ \times \Omega \times \R$ which converges locally uniformly to $F$, and let $(g^\epsilon)_{0 < \epsilon < \epsilon_0}$ be a family of continuous non negative functions on $\Omega$ which converges uniformly to $g$. 

 Assume that for any $0 < \epsilon < \epsilon_0$, $u^\epsilon : \Omega_T \longrightarrow \R$ is a viscosity subsolution (resp. supersolution) to the equation (\ref{eq: CMAF}) for the data $(F^\epsilon,g^\epsilon)$. Then the function $\overline{u}$ (resp. $\underline{u}$) is a viscosity subsolution (resp. supersolution) to the equation \eqref{eq: CMAF} for the data $(F,g)$.
\end{lemma}


The proof below is essentially classical (see \cite{DI04}) but we give a complete account for the reader's convenience. 
\begin{proof}  We prove the statement for supersolutions. The dual arguments work for subsolutions. 

Let $q$ be a lower test function for $\underline{u}$ at $\zeta_0:=(t_0,z_0) \in ]0,T[ \times \Omega$. Fix $r>0$ such that $D_r:= [t_0-r,t_0+r] \times \bar{B}(z_0,r)  \subset \Omega$. By definition there exists a sequence $(\zeta_j)_{j \in \mathbb N}$ in $D_r$ converging to $\zeta_0$ and a sequence $(\varepsilon_j)_{j \in \mathbb N}$ decreasing to $0$  such that
$\lim_{j \to + \infty} u^{\varepsilon_j} (\zeta_j) = \underline{u}(\zeta_0)$.

Fix $\delta >0$ and set 
 $$
 p(z) := q(t,z)-u^{\varepsilon_j}(t,z) - \delta( |z-z_0|^2+ (t-t_0)^2), \ z\in D_r.
 $$

 For each $j \in \mathbb N$ let $w_{j}:= (t_j,z_j)$ be a point in $D_r$ such that $p(w_{j}) = \max_{D_r} p$. We have 
 $$
 q(\zeta_j)-u^{\varepsilon_j}(\zeta_j) - \delta \vert \zeta_j - \zeta_0\vert^2 = p(\zeta_j)  \leq p(w_j) = q(w_j) -u^{\varepsilon_j} (w_j) - \delta |w_j -\zeta_0|^2.
 $$
 
Taking a subsequence if necessary we can assume that $w_j \to w_0 \in D_r$. Then letting $ j \to + \infty$ and taking into account the fact that 
$$
\liminf_{j \to + \infty} u^{\varepsilon_j} (w_j) \geq \underline{u} (w_0),
$$ 
we obtain
$$
q (\zeta_0) - \underline{u} (\zeta_0) \leq q (w_0) - \underline{u} (w_0) - \delta \vert w_0 - \zeta_0\vert^2.
$$
This implies that $\zeta_0 = w_0$, since $q$ is a lower test function for $\underline{u}$ at $\zeta_0$. Hence the sequence $(w_j)$ converges to $\zeta_0$ and then for $j $ large enough $w_j$ is in the interior of $D_r$. By definition of $w_j$,  it follows that for $j$ large enough, the function $ q_j (t,z):= q (t,z) - \delta( |z-z_0|^2+ (t-t_0)^2)$ is a lower test function for $u^{\varepsilon_j}$ at the point $w_j$. Since $u^{\varepsilon_j}$ is a supersolution to the equation (\ref{eq: CMAF}) for the data $(F^{\epsilon_j},g^{\epsilon_j})$, it follows that at the point $w_j = (t_j,z_j)$ we have
 \begin{equation}\label{eq: stability of supersolutions}
 (dd^c q - \delta \beta)_+^n \leq e^{\partial_t q (t_j,z_j)- 2 \delta (t_j-t_0) + F^{\varepsilon_j}(t_{j},z_{j},q(t_{j},z_{j}))} g^{\varepsilon_j}(z_j) dV,
 \end{equation}
 where $\beta =dd^c |z|^2$ is the standard K\"ahler form on $\mathbb{C}^n$.

 We want to prove that at $\zeta_0= (t_0,z_0)$ we have 
 $$
 (dd^c q)_+^n \leq e^{\partial_t q  + F(t_0,z_0,q(t_0,z_0))} g(z_0) dV. 
 $$
 If $dd^c q(z_0)$ has an eigenvalue $\leq 0$ then $(dd^c q)_+^n(z_0)=0$ and the inequality is trivial. If $dd^c q(z_0) > 0$ then letting $j\to +\infty$ and then $\delta\to 0$ in \eqref{eq: stability of supersolutions} we arrive at the desired inequality. 
\end{proof}

\subsection{Comparison of subsolutions}

The main result of this note provides an identification between viscosity and pluripotential subsolutions:

\begin{theorem}\label{thm: pluripotential vs viscosity subsolutions} Let $u\in \mathcal{P}(\Omega_T)\cap L^{\infty}_{\loc}(\Omega_T)$.
The following are equivalent: 
	\begin{itemize}
		\item [(i)] $u$ is a viscosity subsolution to \eqref{eq: CMAF};
		\item [(ii)] $u$ is a pluripotential subsolution to \eqref{eq: CMAF}. 
	\end{itemize}
\end{theorem}

The proof 
relies on corresponding results in the elliptic case, as well as on 
the parabolic comparison principle established in \cite[Theorem 6.5]{GLZ1}.

\begin{proof}  {\it We first prove $(i) \Longrightarrow (ii)$}. 
	Assume $u$ is a viscosity subsolution to \eqref{eq: CMAF}. Fix $J_1 \Subset J_2 \Subset ]0,T[$  compact subintervals.  
	We are going to prove that $u$ is a pluripotential subsolution to \eqref{eq: CMAF} in  $J_1\times \Omega$. 
	
	We regularize $u$ by taking the sup-convolution with respect to the $t$-variable: for $\varepsilon>0$ small enough we define
	$$
	u_{\varepsilon}(t,z) := \sup \left \{ u(t',z) - \frac{1}{2 \varepsilon^2}(t-t')^2  \setdef t' \in J_2\right\}. 
	$$
	The function $u_{\varepsilon}$ is semi-convex in $t \in J_1$, upper semicontinuous in $z$. 
	We claim that  
	$$
	(dd^c u_{\varepsilon})^n \geq e^{\partial_t u_{\varepsilon} + F_{\varepsilon}(t,z,u_{\varepsilon})}gdV,
	$$
	in the viscosity sense where 
	$$
	F_{\varepsilon}(t,z,r) := \inf \left \{ F(t+s,z,r) \setdef  |s| \leq C \varepsilon \right \},
	$$
	for a uniform constant $C>0$ depending on $\sup_{J_2\times \Omega} |u|$. 
The argument	 is classical but we recall it for the reader's convenience.
 Let $q$ be a differential test from above  for $u_{\varepsilon}$ at 
 $(t_0,z_0) \in J_1 \times \Omega$ and let $s_0\in J_2$ be such that 
	$$
	u_{\varepsilon}(t_0,z_0) = u(s_0,z_0)- \frac{1}{2 \varepsilon^2} (s_0-t_0)^2.
	$$ 
	Then $|t_0-s_0|\leq C \varepsilon$. Consider the function $q_{\varepsilon}$ defined by 
	$$
	q_{\varepsilon} (t,z) := q(t+t_0-s_0) + \frac{1}{2 \varepsilon^2} (s_0-t_0)^2.
 	$$
	Then $q_{\varepsilon}(s_0,z_0) = u(s_0,z_0)$, and for all $(t,z) \in J_1\times \Omega$,
	\begin{eqnarray*}
	q_{\varepsilon}(t,z) &\geq &  u_{\varepsilon}(t+t_0-s_0) +  \frac{1}{2 \varepsilon^2} (s_0-t_0)^2 \geq u(t).
	\end{eqnarray*} 
	 In other words, $q_{\varepsilon}$ is a differential test from above for $u$ at $(s_0,z_0)$. Hence 
	$$
	(dd^c q_{\varepsilon})^n (s_0,z_0) \geq e^{\partial_t q_{\varepsilon}(s_0,z_0) +F(s_0,z_0,q_{\varepsilon}(s_0,z_0))} g(z_0) dV. 
	$$
	Since $F$ is increasing in $r$ and $q_{\varepsilon}(s_0,z_0) \geq q(t_0,z_0)$ we obtain
	\begin{eqnarray*}
	(dd^c q)^n (t_0,z_0)& \geq & e^{\partial_t q(t_0,z_0) +F(s_0,z_0,q(t_0,z_0))} g(z_0) dV\\
	& \geq & e^{\partial_t q(t_0,z_0) +F_{\varepsilon}(t_0,z_0,q(t_0,z_0))} g(z_0) dV,
	\end{eqnarray*} 
	as claimed.
	
	Let $\partial^-_t u_{\varepsilon}$ denote the left derivative in $t$ of $u_{\varepsilon}$.  Since $\partial^-_t u_{\varepsilon} + F_{\varepsilon}$ is bounded, 
	by considering $u_{\varepsilon}+ \delta |z|^2$ and letting $\delta\to 0$, we can assume that $g\geq c>0$ is strictly positive in 
	$\Omega$. The function
	$$
	(t,z) \mapsto G(t,z)=e^{\partial_t^- u_{\varepsilon}(t,z) +F_{\varepsilon}(t,z,u_{\varepsilon}(t,z))}g(z),
	$$
is lower semicontinuous in $\Omega_T$. 
 It can be approximated from below by a sequence of positive continuous functions $(G_j)$. By definition of viscosity subsolutions (applied to $u_{\varepsilon}$) we have 
\begin{equation}\label{eq: pluri vs visco subsolutions}
 (dd^c u_{\varepsilon})^n \geq G_jdV
 \end{equation}
 in the parabolic viscosity sense. Since $G_j$ is continuous, we  can thus invoke  \cite[Proposition 3.6]{EGZ15} to conclude that \eqref{eq: pluri vs visco subsolutions}  holds in the elliptic viscosity sense for each $t\in J_1$ fixed. It then follows from \cite[Proposition 1.5]{EGZ11} that \eqref{eq: pluri vs visco subsolutions} holds in the elliptic pluripotential sense for each $t\in J_1$ fixed. Now, \cite[Proposition 3.2]{GLZ1} ensures that $u_{\varepsilon}$ is a parabolic pluripotential subsolution to \eqref{eq: CMAF}. Since $u_{\varepsilon}$ decreases to $u$, Lemma \ref{lem: convergence of subsolutions} insures that $u$ is a pluripotential subsolution to \eqref{eq: CMAF}.

%
%
%
%
	 \medskip
	 
	 {\it We now prove $(ii) \Longrightarrow (i)$.}
	Assume  that $u$ is a pluripotential subsolution to \eqref{eq: CMAF}. 
	 Fix $(t_0,z_0) \in \Omega_T$ and $q$ a differential test from above defined in a neighborhood $J\times U \Subset ]0,T[ \times \Omega$ of $(t_0,z_0)$. We need to prove that 
	\begin{equation}
		\label{eq:pluri is visco}
		(dd^c q)^n(t_0,z_0) \geq e^{\partial_t q(t_0,z_0) +F(t_0,z_0,q(t_0,z_0))}g(z_0)dV.
	\end{equation}
	It follows from \cite{EGZ11} that $dd^c q$ is semipositive at $(t_0,z_0)$.
	If $g(z_0)=0$  the inequality  follows from the elliptic theory (see \cite{EGZ11}). Since $g$ is continuous up to shrinking $U$, we can   assume that $g>0$ in $U$. 
	
	Assume by contradiction that \eqref{eq:pluri is visco} does not hold. Then, by continuity of the functions involved, there exists $\varepsilon,r,\delta>0$ small enough such that 
	$$
	(dd^c q + \varepsilon dd^c |z|^2)^n < e^{\partial_t q(t,z) +F(t,z,q(t,z)) - \delta}g(z) dV
	$$
	holds in the classical sense in $[t_0-r,t_0+r] \times B(z_0,r)$. Consider the function 
	$$
	v(t,z) := q(t,z) + \gamma (|z-z_0|^2-r^2+t_0-t), 
	$$
	for $(t,z) \in [t_0-r,t_0]\times B(z_0,r)$. For $\gamma$ small enough one can check that 
	\begin{flalign*}
		(dd^c v)^n & \leq  e^{\partial_t q(t,z) + F(t,z,q(t,z)) - \delta} g(z)dV \\
		& \leq e^{\partial_t v + F(t,z,v+ \gamma r^2 + \gamma(t-t_0)) + \gamma -\delta} g(z)dV\\
		& \leq e^{\partial_t v+ F(t,z,v)} g(z)dV,
	\end{flalign*}
	hence $v$ is a supersolution to \eqref{eq: CMAF} in $]t_0-r,t_0[\times B(z_0,r)$. We next compare $v$ and $u$ on the parabolic boundary of $]t_0-r,t_0[ \times B(z_0,r)$. For all $z\in B(z_0,r)$ we have
	$$
	v(t_0-r,z) \geq  q(t_0-r,z) + \gamma (r-r^2) \geq q(t_0-r,z)  \geq u(t_0-r,z), 
	$$
	if $r<1$. For all $t \in [t_0-r,t_0], \zeta \in \partial B(z_0,r)$ we have 
	$$
	v(t,\zeta) = q(t,\zeta) + \gamma(t_0-t) \geq u(t,\zeta). 
	$$
		
	If $u$ is locally uniformly Lipschitz in t, it follows from \cite[Theorem 6.5]{GLZ1} 
	that $u\leq v$ in $[t_0-r,t_0] \times B(z_0,r)$. This yields a contradiction as
	$$
	v(t_0,z_0) =q(t_0,z_0) -\gamma r^2 < u(t_0,z_0).  
	$$
	
	We finally remove the Lipschitz assumption on $u$.   For each $\varepsilon>0$ we define $u_{\varepsilon}$ by
	$$
	u_{\varepsilon} (t,z) := \int_{\mathbb{R}} u(st, z) \chi((s-1)/\varepsilon) ds,
	$$
	where $\chi$ is a cut-off function. Let $F_{j}$ be a family of smooth functions which increases to $F$. Then $u$ is a pluripotential subsolution to \eqref{eq: CMAF} with data $F_{j}$. Arguing as in \cite[Theorem 6.5]{GLZ1} we can show that $u_{\varepsilon} -c(\varepsilon)(t+1)$ is a pluripotential subsolution to \eqref{eq: CMAF} (with data $F_{j}$) which is locally uniformly Lipschitz. Hence, we can apply the first step to show that  $u_{\varepsilon} -c(\varepsilon)(t+1)$ is a viscosity subsolution to \eqref{eq: CMAF} with data $F_{j}$.  Thanks to Lemma \ref{lem:Stability} we can let $\varepsilon \to 0$ and then $j\to +\infty$ to conclude the proof. 
\end{proof}

\section{Viscosity vs pluripotential (super)solutions}

The notion of pluripotential supersolutions has been introduced in \cite{GLZ1}.  
In case $u\in \mathcal{P}(\Omega_T)\cap L^{\infty}_{\loc}(\Omega_T)$ is locally uniformly semiconcave, it is a pluripotential supersolution to \eqref{eq: CMAF} if 
$$
(dd^c u)^n \wedge dt\leq e^{\partial_t^{-} u +F(t,z,u)}gdV \wedge dt,
$$
in the sense of Radon measures in $\Omega_T$.

As in the viscosity setting, a {\it pluripotential solution} is a parabolic potential which is both a subsolution and a supersolution.

\subsection{Comparison of supersolutions}

\begin{theorem}\label{thm: pluri visco super}
	Assume $v\in \mathcal{P}(\Omega_T)\cap C(\Omega_T)$ is a pluripotential supersolution to \eqref{eq: CMAF} which
 is locally uniformly semi-concave in $t\in ]0,T[$.  Then $v$ is a viscosity supersolution to \eqref{eq: CMAF}. 
\end{theorem}

The proof relies on the parabolic pluripotential comparison principle \cite[Theorem 6.5]{GLZ1} which requires the extra 
semi-concavity hypothesis.

\begin{proof}
We can assume that $g>0$. 
	Fix $(t_0,z_0)\in \Omega_T$ and let $q$  be a 
differntial test from below	
 for $v$ at $(t_0,z_0)$, defined in $J\times U\Subset \Omega_T$. We want to prove that 
	\begin{equation}
		\label{eq:pluri is visco super}
		(dd^c q)_{+}^n(t_0,z_0) \leq e^{\partial_t q(t_0,z_0) +F(t_0,z_0,q(t_0,z_0))}g(z_0)dV.
	\end{equation}
	Assume, by contradiction, that it is not the case. Then $dd^c q_{t_0}(z_0)$ is semipositive and there is a constant $\delta>0$ such that 
	$$
	(dd^c q_{t_0}(z_0))^n > e^{\partial_t q(t_0,z_0) + F(t_0,z_0,q(t_0,z_0))+ 2\delta} g(z_0)dV(z_0). 
	$$
	Since $g>0$ and the data is continuous, we can find $r\in ]0,1[$ so small that 
	$$
	(dd^c q - \varepsilon dd^c |z|^2)^n \geq e^{\partial_t q(t,z) + F(t,z,q(t,z)) +\delta}g(z)dV(z)
	$$
	holds in the classical sense in $[t_0-r,t_0+r] \times B(z_0,r)$. Consider the function 
	$$
	u(t,z) := q(t,z) - \gamma (|z-z_0|^2-r^2+t_0-t), 
	$$
	for $(t,z) \in [t_0-r,t_0]\times B(z_0,r)$. For $\gamma$ small enough one can check that 
	\begin{flalign*}
		(dd^c u)^n & \geq  e^{\partial_t q(t,z) + F(t,z,q(t,z)) + \delta} g(z)dV \\
		& \geq e^{\partial_t u -\gamma + F(t,z,u- \gamma r^2 + \gamma(t_0-t))  -\delta} g(z)dV\\
		& \geq e^{\partial_t u+ F(t,z,u)} g(z)dV,
	\end{flalign*}
	hence $u$ is a subsolution to \eqref{eq: CMAF} in $]t_0-r,t_0[\times B(z_0,r)$. We next compare $v$ and $u$ on the parabolic boundary of $]t_0-r,t_0[ \times B(z_0,r)$. For all $z\in B(z_0,r)$ we have
	$$
	u(t_0-r,z) \leq  q(t_0-r,z) + \gamma (r^2-r) \leq q(t_0-r,z)  \leq v(t_0-r,z), 
	$$
	since $r<1$. For all $t \in [t_0-r,t_0], \zeta \in \partial B(z_0,r)$ we have 
	$$
	u(t,\zeta) = q(t,\zeta)- \gamma(t_0-t) \leq v(t,\zeta). 
	$$
	Since $v$ is locally uniformly semi-concave, we can invoke \cite[Theorem 6.5]{GLZ1} to conclude that $u\leq v$ in $[t_0-r,t_0] \times B(z_0,r)$. This yields a contradiction since
	$
	u(t_0,z_0) =q(t_0,z_0) +\gamma r^2 > v(t_0,z_0).  
	$
\end{proof}

In the reverse direction we have the following observation:

\begin{theorem}
Let $v$ be a  viscosity supersolution to \eqref{eq: CMAF} and assume that $v$ is locally uniformly semi-concave in $t\in ]0,T[$.  Then $P(v)$ is a pluripotential supersolution to \eqref{eq: CMAF}. 
\end{theorem}

Here $P(v)(t,z)=P(v_t)(z)$ is the slice plurisubharmonic envelope of $v$:
for each $t$ fixed, we set
$$
P(v_t)(z):=\sup \{ w(z); \; w \leq v_t \; \text{ and } w \text{ plurisubharmonic in } \Omega \},
$$
 i.e. $P(v)_t:= P(v_t)$ is the largest psh function lying below $v_t$.

\begin{proof}
We first observe that $t \mapsto P(v)(t,z)$ is locally 
uniformly semi-concave. This follows from the fact that 
$v \mapsto P(v)$ is increasing and concave: assume for simplicity that
$t \mapsto v(t,z)$ is uniformly concave, then
$$
\frac{v_{t+s}+v_{t-s}}{2} \leq v_t \Rightarrow 
\frac{P(v_{t+s})+P(v_{t-s})}{2} \leq P\left( \frac{v_{t+s}+v_{t-s}}{2} \right)  \leq P(v_t).
$$

Fix $U\Subset \Omega$ and $S\in ]0,T[$. Let $v^{\varepsilon}$ denote the inf-convolution of $v$.
Then $v^{\varepsilon}$ increases pointwise to $v$ and $P(v^{\varepsilon}) \uparrow P(v)$ as $\varepsilon\downarrow 0$.  Since $\partial_t P(v^{\varepsilon})$ converges a.e.  to $\partial_t P(v)$ (see \cite{GLZ1}), it suffices to prove that each $P(v^{\varepsilon})$ is a pluripotential supersolution to \eqref{eq: CMAF}. 
We can thus assume that $v$ is continuous in $\Omega_T$. 

	The left derivative $\partial^-_t v$ is upper semicontinuous in $\Omega_T$. It follows from \cite[Proposition 3.6]{EGZ15} that, for all $t\in ]0,T[$, the inequality 
	$$
	(dd^c v_t)_+^n \leq e^{\partial^-_t v+ F(t,\cdot,v_t)} g dV
	$$
	holds in the viscosity sense in $\Omega$. It thus follows from \cite{GLZ17} that $P(v_t)$ satisfies 
	$$
	(dd^c P(v_t))^n \leq e^{\partial^-_t v + F(y,\cdot, P(v_t))} gdV
	$$
	in the pluripotential sense. Set 
	$$
	E =\{ (t,z) \in \Omega_T, \; 
	\partial^+_t v(t,z) = \partial^-_t v(t,z)
	\; \& \; 
	\partial^+_t P(v)(t,z) = \partial^-_t P(v)(t,z) \}.
	$$
Then $\Omega_T \setminus E$ has zero Lebesgue measure. If $(t,z) \in E \cap \{P(v_t)=v_t\}$ then  $\partial^-_t P(v)(t,z) =\partial^-_t v(t,z)$. Therefore, 
	$$
 (dd^c P(v))^n \wedge dt \leq e^{\partial_t P(v) +F(t,z,P(v))} g   dV(z) \wedge dt
	$$
	holds in the pluripotential sense in $\Omega_T$. 
\end{proof}

\subsection{Viscosity comparison principle}

The following stability estimate follows directly from the  viscosity comparison principle established in  \cite[Theorem B]{EGZ15}. 

\begin{lemma}\label{lem: VCP}
Assume $u$ is a  bounded viscosity subsolution to \eqref{eq: CMAF} with data $F$ and $v$ is a bounded viscosity supersolution to \eqref{eq: CMAF} with data $G$. Then 
$$
\sup_{\Omega_T} (u-v) \leq \sup_{\partial_0 \Omega_T} (u^*-v_{*})_+  + T \Vert (G-F)_{+}\Vert,
$$
where $\Vert (F-G)_{+}\Vert := \max_{[0,T] \times \bar{\Omega} \times [-C_0,+ C_0]} (F-G)_{+} $ and $C_0 > 0$ is a uniform bound on $\vert u\vert$ and $\vert v\vert$ in $\Omega_T$.
\end{lemma}

\begin{proof}
Set 
$$
M_1:= \sup_{\partial_0 \Omega_T} (u^*-v_{*})_+ , \ M_2 :=  \Vert (G-F)_{+}\Vert,
$$
and $\tilde{u}:= u -M_1- M_2t$. Then $\tilde{u}^*\leq v^*$ on $\partial_0 \Omega_T$. It follows directly from the definition of viscosity subsolutions that $\tilde{u}$ is a viscosity subsolution to \eqref{eq: CMAF} with data $G$ since $F + (G-F)_+ \geq G$. It thus follows from \cite[Theorem B]{EGZ15} that $\tilde{u}\leq v$, giving the desired estimate. 

\end{proof}

\begin{coro} \label{cor: stability} Assume that $F^{j} \to F$ locally uniformly in $\Omega_T \times \R$. 
Let $h^{j}$ be a sequence of parabolic boundary data converging locally uniformly to a parabolic boundary datum $ h$ on $\partial_0 \Omega$.

Let $\phi^{j}$ be the unique viscosity solution to the Cauchy Dirichlet problem for the data $(F^{j},g,h^{j})$.
Then $(\phi^j)_{j \in \NN}$ converges locally uniformly in $\Omega_T$ to a continuous function $ \phi$ which is the unique viscosity solution to the Cauchy-Dirichlet problem of the equation (\ref{eq: CMAF}) for the data $(F,g,h)$.

\end{coro}
\begin{proof}
By the viscosity comparison principle (Lemma \ref{lem: VCP})  we have for $j, k \in \NN$, for any $0 < S < T$,
$$
\sup_{\bar{\Omega}_S} \vert \phi_j - \phi_k\vert \leq \sup_{\partial_0 \Omega_S} \vert h_j - h_k\vert + S\Vert F^j - F^k\Vert_{\bar{\Omega}_S \times L},
$$ 
where $L \subset \R$ is a compact set containing the values of $\phi^j$, $j\in \mathbb{N}$, on the compact set $\bar{\Omega}_S$.
It follows that $(\phi_j)$ is a Cauchy sequence for the norm of the uniform convergence on each $\bar{\Omega}_S$. Then the sequence has a limit which is a continuous function $\phi : [0,T[ \times \bar{\Omega}$.
By  Lemma \ref{lem:Stability}, the function  $\phi$ is a solution to the equation (\ref{eq: CMAF}) for the data $(F,g,h)$.
Set 
$$
\alpha_j := \sup_{\partial_0 \Omega_S} \vert h_j - h_k\vert + S\Vert F^j - F^k\Vert_{\bar{\Omega}_S \times L}.
$$
Then $\alpha_j \to 0$ and for $j >> 1$ we have
$$ \phi_j - \alpha_j \leq \phi \leq \phi_j + \alpha_j,
$$
in $\Omega_S$. From this inequality it follows that the boundary values of $\phi$ coincide with $h$ on $\partial_0 \Omega_S$.
Letting $S \to T$, we see that $\phi$ is the unique solution to the equation (\ref{eq: CMAF}) for the data $(F,g,h)$.
\end{proof}

\subsection{Viscosity vs pluripotential solutions}
If $h$ does not depend on $t$,  it was shown in \cite{EGZ15} that there exists a unique viscosity solution to \eqref{eq: CMAF}  with boundary value $h$. 
This  is the Perron envelope of all viscosity subsolutions with boundary value $h$.

This result has been recently extended by Do-Le-T\^o \cite{DLT19} to boundary data that are time-dependent.
Combining viscosity and pluripotential techniques we provide an alternative proof of this existence result:

%
%

\begin{theorem}\label{thm: viscosity solution is pluripotential solution}
	The Perron envelope of viscosity subsolutions to \eqref{eq: CMAF} with boundary value $h$ is the unique viscosity solution to \eqref{eq: CMAF} with boundary value $h$. It coincides with the envelope of all pluripotential subsolutions to  \eqref{eq: CMAF} with boundary value $h$.
\end{theorem}

\begin{proof}
We first assume that the data $(h,F)$ satisfiy the assumptions of \cite{GLZ1}. 
	Let $U$ be the envelope of all pluripotential subsolutions to \eqref{eq: CMAF} with boundary value $h$, and $V$ be the Perron envelope of  viscosity subsolutions to  \eqref{eq: CMAF} with boundary value $h$.  Theorem \ref{thm: pluripotential vs viscosity subsolutions} ensures that $U=V$. By Proposition \ref{prop: identification} and \cite{GLZ1}, $U \in \mathcal{C}(\Omega_T)$ is a pluripotential solution to \eqref{eq: CMAF} which is locally uniformly semi-concave. It then follows from Theorem \ref{thm: pluri visco super} that $U$ is a viscosity supersolution to \eqref{eq: CMAF}, hence $U$ is a viscosity solution to \eqref{eq: CMAF}.  Lemma \ref{lem: VCP} ensures that $U$ is the unique viscosity solution to \eqref{eq: CMAF} with boundary value $h$.


	We now treat the general case. Let $(h_j,F_j)$ be approximants of $(h,F)$ which satisfy the assumptions in \cite{GLZ1}, and let $U_j$ be the envelope of pluripotential subsolutions to \eqref{eq: CMAF} with data $(h_j,F_j)$. Then $U_j$ is a pluripotential solution to \eqref{eq: CMAF} which is locally uniformly semiconcave. The previous step ensures that $U_j$ is a viscosity solution to \eqref{eq: CMAF} with data $(h_j,F_j)$. 
	By stability of viscosity solutions (see Lemma \ref{lem: VCP}), $U_j$ uniformly converges to $U$ and $U=h$ on $\partial_0 \Omega_T$. By Corollary \ref{cor: stability}, $U$ is a solution to the equation \eqref{eq: CMAF} in $\Omega_T$. Hence $U$ is a solution to the Cauchy-Dirichlet problem for \eqref{eq: CMAF} in $\Omega_T$ with boundary values $h$.
	
	Uniqueness follows from the viscosity  comparison principle in Lemma \ref{lem: VCP} (see \cite[Theorem B]{EGZ15}). 
\end{proof}

\section{Compact K\"ahler manifolds}\label{sect: compact case}

The techniques developed in the local context allow us to obtain analogous results in the compact setting. 

We consider the following complex Monge-Amp\`ere flow 
\begin{equation}\label{eq: CMAF global}
	(\omega_t +dd^c \varphi_t)^n = e^{\dot \varphi_t + F(t,x,\varphi_t)} gdV,
\end{equation}
where $X$ is a compact K\"ahler manifold of dimension $n$ and
\begin{enumerate}
	\item  $X_T:= ]0,T[\times X$ with $T>0$;
	\item $0  < g$ is a continuous function on $X$; 
	\item   $t \mapsto \omega(t,x)$ is a smooth family of  
	closed  semi-positive $(1,1)$-forms 
  such that  $ \theta(x) \leq \omega_t(x)\leq \Theta$,
where $\theta$ is a closed semi-positive big form, and $\Theta$ is a K\"ahler form;
  \item $(t,x,r) \mapsto F (t,x,r)$ is continuous in $[0,T[ \times X \times \R$, increasing  in $r$; 
  \item   $\f : [0,T[ \times X \rightarrow \R$ is the unknown function, with $\f_t: = \f (t,\cdot)$. 
   \end{enumerate}
   
   Let $\varphi_0$ be a bounded $\omega_0$-psh function on $X$ which is continuous in $\Omega$, the ample locus of $\{\theta\}$.   

\begin{definition}
The set $\mathcal{P}(X_T,\omega_t)$ of parabolic potentials consists of functions $u: X_T \rightarrow \mathbb{R}\cup \{-\infty\}$ such that 
\begin{itemize}
\item $u$ is upper semi-continuous on $X_T$ and $u\in L^1_{\loc}(X_T)$; 
\item for each $t\in ]0,T[$, the function $u_t:= u(t,\cdot)$ is $\omega_t$-psh on $X$.  
\end{itemize}
\end{definition}

\begin{definition}
A parabolic potential $u\in \mathcal{P}(X_T,\omega_t)\cap L^{\infty}(X_T)$ is a pluripotential subsolution to \eqref{eq: CMAF global} if for all constant $a>0$,   
$$
(\omega_t+ dd^c u_t)^n \wedge dt \geq  g (a (\partial_t \varphi + F(t,z,u_t(z)) - a\log a +a) \, dV(z) \wedge dt
$$
holds in the sense of distribution in $X_T$.  
\end{definition}
If $u\in \mathcal{P}(X_T,\omega_t)\cap L^{\infty}(X_T)$ is locally uniformly Lipschitz in $t$ then our definition coincides with that of \cite{GLZ2}.




\begin{theorem}
Let $U$ (respectively $V$) be the envelope of all pluripotential (respectively viscosity) subsolutions $u$ to \eqref{eq: CMAF global} such that $\limsup_{t\to 0} u_t \leq \varphi_0$. Then $U=V$ is the unique viscosity solution to \eqref{eq: CMAF global} starting from $\varphi_0$. 
\end{theorem}

The last condition in the theorem means that $\lim_{t\to 0^+} U_t =\varphi_0$  locally uniformly in $\Omega := {\rm Amp} (\{\theta\})$, the ample locus of the class 
$\{\theta\}$. 

\begin{proof}
The equivalence of pluripotential and viscosity subsolutions for a given parabolic potential $u\in \mathcal{P}(X_T)\cap L^{\infty}(X_T)$ follows from Theorem \ref{thm: pluripotential vs viscosity subsolutions},  since being a pluripotential (resp. viscosity) subsolution is a local property. It follows in particular that $U = V$ on $X_T$.

We approximate $F$ uniformly by a sequence of data $F^j$ which satisfy the assumptions in \cite{GLZ2} (one can e.g. take the convolution with a smoothing kernel in $t,r$). We approximate $\omega_t$ by $\omega_t^{j}:= \omega_t +2^{-j} \Theta$. Then $\omega^j$ also satisfies the assumptions in \cite{GLZ2}.   Let $U^j$ be the envelope of pluripotential subsolutions to \eqref{eq: CMAF} with data $(F^j,\omega^j, \varphi_0)$. By \cite{GLZ2} and the proof of Proposition \ref{prop: identification}, $U^j$ is locally uniformly semi-concave in $t$,  $\lim_{t\to 0^+} U_t^j = \varphi_0$, for all $j$, and $U^j$ is a pluripotential solution to \eqref{eq: CMAF} with data $(F^j,\omega^j)$. By continuity of $\varphi_0$ in $\Omega$ and \cite[Proposition 2.2]{GLZ2}, we infer that $U_t^j$ locally uniformly converges to $\varphi_0$ in $\Omega$. 

The proof of Theorem \ref{thm: pluri visco super} shows that $U^j$ is a viscosity solution to \eqref{eq: CMAF global} in $\Omega$. We now prove that $U^j$ locally uniformly converges to $U$ on $\Omega_T$. If we can do this then $U \in \mathcal{C}(\Omega_T)$ is a viscosity solution to \eqref{eq: CMAF global}   (thanks to Lemma \ref{lem:Stability}), and $\lim_{t\to 0^+} U_t = \varphi_0$ locally uniformly in $\Omega$.  

In the arguments below we use $\varepsilon(j)$ to denote various positive constants which tend to $0$ as $j\to +\infty$. 

Since $\omega\leq \omega^j$, the function $U- \varepsilon(j)t$ is a pluripotential subsolution to \eqref{eq: CMAF global} with datum $(F^j,\omega^j)$, hence 
\begin{equation}\label{eq: bound 1}
U-\varepsilon(j)t \leq U^j. 
\end{equation}
To obtain the other bound we fix $\rho \in \PSH(X,\theta)\cap L^{\infty}(X)$, $\sup_X \rho =0$, such that 
$$(\theta +dd^c \rho)^n= 2 ^n e^{c_1} gdV,$$ for some constant $c_1\in \mathbb{R}$. The existence of $\rho$ follows from \cite{EGZ09}. 
Let $\psi\leq 0$ be a $\theta$-psh function which is smooth in $\Omega$ and satisfies 
$$
\theta+dd^c \psi \geq 2c_0 \Theta,
$$
for some positive fixed constant $c_0$. 

Set for $j \in \mathbb N$,
 \begin{eqnarray*}
 W^j := (1-\lambda_j) U^j  + \lambda_j \frac{\rho + \psi}{2}, \, \, \, \, \text{with} \, \, \, \, \lambda_j := \frac{2^{- j}}{2^{- j} + c_0}. 
 \end{eqnarray*}
Given this choice of $\lambda_j$, a direct computation shows that 
\begin{eqnarray*}
\omega_t +dd^c W^j &\geq & (1-\lambda_j) (\omega_t + dd^c U^j_t) + \lambda_j (\omega_t + dd^c ((\rho + \psi) \slash 2)) \\
&\geq &(1-\lambda_j) (\omega_t^j + dd^c U^j_t)  + \lambda_j  (\theta +dd^c \rho)\slash 2\geq 0. 
\end{eqnarray*}
Hence, applying \cite[Lemma 3.15]{GLZ2} we obtain 
\begin{eqnarray*}
(\omega_t +dd^c W^j)^n &\geq & e^{(1-\lambda_j) (\partial_t U^j_t + F^j(t,x, U^j)) + \lambda_j c_1 } gdV\\
& \geq & e^{\partial_t W^j  + F(t,x,W^j) -\varepsilon'(j)} g dV, 
\end{eqnarray*}
in the weak sense on $\Omega$, where $\varepsilon' (j) \to 0$.

It thus follows that $W^j-\varepsilon(j)t$ is a pluripotential subsolution to the equation \eqref{eq: CMAF global} on $\Omega_T$ with datum $(F, \omega)$. 
Observe that $W^j$ is not bounded on $X$. Since, for $C$ large enough $u:=\rho + nt \log t-Ct-C$ is a bounded pluripotential subsolution  to the equation \eqref{eq: CMAF global} in $X_T$ with datum $(F, \omega)$, it follows that $\tilde{W}^j := \sup\{W^j - \varepsilon(j)t, u\}$ is a bounded subsolution to the \eqref{eq: CMAF global} on $X_T$.
Since $W^j (t,x) \leq  U^j (t,x) + \varepsilon''_j$  where $\varepsilon''_j \to 0$, and $\lim_{t \to 0} U^j (t,x) = \varphi_0 (x)$ for any $x \in X$, it follows that 
\begin{equation}\label{eq: bound 2}
\tilde{W}^j  - \varepsilon''_j \leq U, \, \, \, \text{in} \, \, \, X_T. 
\end{equation}

 From \eqref{eq: bound 1} and \eqref{eq: bound 2} we conclude that $U^j$ locally uniformly converges to $U$ on $X_T$. 

The uniqueness follows from \cite{To19}. 
\end{proof}


\begin{thebibliography}{99}

\bibitem[BT76]{BT87}
Eric Bedford and B.~A. Taylor, \emph{The {D}irichlet problem for a complex
  {M}onge-{A}mp\`ere equation}, Invent. Math. \textbf{37} (1976), no.~1, 1--44.
  

\bibitem[BT82]{BT82}
Eric Bedford and B.~A. Taylor, \emph{A new capacity for plurisubharmonic
  functions}, Acta Math. \textbf{149} (1982), no.~1-2, 1--40. 
  
    \bibitem[DI04]{DI04} J\'er\^ome Droniou, Cyril Imbert,  {\em Solutions de viscosit\'e et solutions variationnelles pour EDP non lin\'eaires}, Cours de D.E.A, Universit\'e de Montpellier, 2004.
  
  \bibitem[DLT19]{DLT19} Hoang-Son Do, Giang Le, and Tat Dat To,
  \emph{Viscosity  solutions to parabolic complex {M}onge-{A}mp\`ere equations}.  arXiv:1905.11818.

 \bibitem[EGZ09]{EGZ09}  Philippe Eyssidieux, Vincent Guedj, and Ahmed Zeriahi,  {\em Singular K\"ahler-Einstein metrics},
 J. Amer. Math. Soc. {\bf 22} (2009), 607--639. 


\bibitem[EGZ11]{EGZ11} Philippe Eyssidieux, Vincent Guedj, and Ahmed Zeriahi, \emph{Viscosity solutions to  degenerate complex Monge-Amp\`ere equations.} Comm. Pure Appl. Math. 64 (2011), no. 8, 1059--1094. 

\bibitem[EGZ15]{EGZ15} Philippe Eyssidieux, Vincent Guedj, and Ahmed Zeriahi, \emph{Weak solutions to
  degenerate complex {M}onge-{A}mp\`ere flows {I}}, Math. Ann. \textbf{362}
  (2015), no.~3-4, 931--963.  

\bibitem[EGZ16]{EGZ16} Philippe Eyssidieux, Vincent Guedj, and Ahmed Zeriahi, \emph{Weak solutions to  degenerate complex {M}onge-{A}mp\`ere flows {II}}, Adv. Math. \textbf{293} (2016), 37--80.  

\bibitem[EGZ17]{EGZ17} Philippe Eyssidieux, Vincent Guedj, and Ahmed Zeriahi, \emph{Corrigendum: Viscosity solutions to complex Monge-Ampère equations}. Comm. Pure Appl. Math. 70 (2017), no. 5, 815--821.

\bibitem[EGZ18]{EGZ18} Philippe Eyssidieux, Vincent Guedj, and Ahmed Zeriahi, \emph{Convergence of weak K\"ahler-Ricci flows on minimal models of positive Kodaira dimension.} Comm. Math. Phys. 357 (2018), no. 3, 1179--1214.

\bibitem[GLZ17]{GLZ17} Vincent Guedj, Chinh~H. Lu, and Ahmed Zeriahi, \emph{Plurisubharmonic envelopes
  and supersolutions}, arXiv:1703.05254, to appear in J. Differential Geom.


\bibitem[GLZ1]{GLZ1}
Vincent Guedj, Chinh~H. Lu, and Ahmed Zeriahi, \emph{The pluripotential
  Cauchy-Dirichlet problem for complex Monge-Amp\`ere flows},  arXiv:1810.02122. 

\bibitem[GLZ2]{GLZ2}
Vincent Guedj, Chinh~H. Lu, and Ahmed Zeriahi, \emph{Pluripotential
  K\"ahler-Ricci flows},  arXiv:1810.02121. 

\bibitem[HL13]{HL13}
F.~Reese Harvey and H.~Blaine Lawson, Jr., \emph{The equivalence of viscosity
  and distributional subsolutions for convex subequations---a strong {B}ellman
  principle}, Bull. Braz. Math. Soc. (N.S.) \textbf{44} (2013), no.~4,
  621--652.  

\bibitem[To19]{To19}  Tat Dat To,
  \emph{Convergence of the weak K\"ahler-Ricci flow on manifolds of general type},	arXiv:1905.01276.
\end{thebibliography}
\end{document}